\documentclass{amsart}                                                 
\usepackage[T1]{fontenc}
\usepackage[english]{babel}
\usepackage{amsmath}
\usepackage{amsthm}                      
\usepackage{amssymb}
\usepackage{url}
\usepackage[small,it]{caption}
\theoremstyle{plain}                                                           
\newtheorem{thm}{Theorem}[section]
\newtheorem{lem}[thm]{Lemma}

\newtheorem{prop}[thm]{Proposition}
\theoremstyle{definition}
\newtheorem{ex}[thm]{Example}
\newtheorem{rem}[thm]{Remark}
\newtheorem{convention}[thm]{Convention}
\newtheorem{defn}[thm]{Definition}
\newtheorem{notation}[thm]{Notation}
\title{Cohomology of local systems on loci of $d$-elliptic abelian surfaces}
\author{Dan Petersen}

\thanks{The author is supported by the G\"oran Gustafsson Foundation for Research in Natural Sciences and Medicine.}
\address{Dan Petersen\\Department of Mathematics\\ KTH Royal Institute of Technology\\ 100 44 Stockholm \\ Sweden} 
\DeclareMathOperator{\id}{id}  

\DeclareMathOperator{\SL}{SL}
\DeclareMathOperator{\GL}{GL}

\DeclareMathOperator{\Sym}{Sym}                  
\newcommand{\Sp}{\mathrm{Sp}}

\DeclareMathOperator{\Res}{Res}

\newcommand{\la}{\langle}
\newcommand{\ra}{\rangle}
\newcommand{\field}[1]{\ensuremath{\mathbf{#1}}}
\newcommand{\R}{\ensuremath{\field{U}}}        
\newcommand{\Q}{\ensuremath{\field{Q}}}
 \newcommand{\sym}{\ensuremath{\mathbb{S}}}
 \newcommand{\Z}{\ensuremath{\field{Z}}} 
                                                           
\newcommand{\C}{\ensuremath{\field{C}}}  
\newcommand{\V}{\ensuremath{\field{V}}}

\newcommand{\W}{\ensuremath{\field{W}}}
\newcommand{\e}{\ensuremath{\mathbf{e}}}
\newcommand{\A}{\ensuremath{\mathcal{A}}}
\newcommand{\tate}{\ensuremath{\field{L}}}
\newcommand{\hej}{\ensuremath{\mathcal{E}}}
\newcommand{\hh}{\ensuremath{\mathcal{E}}}
\newcommand{\M}{\ensuremath{\mathcal{M}}}  
\newcommand{\X}{\ensuremath{\mathcal{X}}}
\newcommand{\Y}{\mathcal{Y}}

\newcommand{\st}{\; | \;}            
\begin{document} 
   \maketitle

\begin{abstract}
We consider the loci of $d$-elliptic curves in $\mathcal{M}_2$, and corresponding loci of $d$-elliptic surfaces in $\mathcal{A}_2$. We show how a description of these loci as quotients of a product of modular curves can be used to calculate cohomology of natural local systems on them, both as mixed Hodge structures and $\ell$-adic Galois representations. We study in particular the case $d=2$, and compute the Euler characteristic of the moduli space of $n$-pointed bi-elliptic genus $2$ curves in the Grothendieck group of Hodge structures. \end{abstract}

\section{Introduction}

To an irreducible representation of $\Sp_{2g}$ with highest weight vector $\lambda$ one can associate in a natural way a local system $\W_\lambda$ on the moduli spaces $\mathcal{A}_g$, hence also $\M_g$. One reason for studying these local systems is that their complex (resp. $\ell$-adic) cohomology groups will contain spaces of elliptic and Siegel modular forms (resp. their associated $\ell$-adic Galois representations) as subquotients. In particular, one can study modular forms by looking at the cohomology of these local systems, and vice versa. 

When $g=1$ this is described by the Eichler--Shimura theory, and in particular its Hodge-theoretic/$\ell$-adic interpretation \cite{deligne69} which expresses the cohomology of such a local system in terms of spaces of modular forms on the corresponding modular curve. See \cite[\S 4]{faltings87} for a r\'esum\'e. For higher genera the situation is not as well understood. The (integer-valued) Euler characteristics of these local systems on $\M_2$ were calculated in \cite{getzler02}. Their Euler characteristics on $\M_g$ and $\A_g$ for $g=2,3$, now taken in the Grothendieck group of $\ell$-adic Galois representations, have been investigated by means of point counting in the sequence of papers \cite{fvdg1}, \cite{fvdg2}, \cite{bfg08}, \cite{bfg11}.

Another reason to be interested in such local systems is that they arise when computing the cohomology of relative configuration spaces. For instance, in the case of $\M_g$, the results of \cite{getzler99} imply that calculating the Euler characteristics of all of these local systems on $\mathcal M_g$ is equivalent to calculating the $\sym_n$-equivariant Euler characteristic of $\M_{g,n}$ for all $n$. 

In this article, we shall study the restriction of these local systems to certain loci in $\A_2$ of abelian surfaces with a degree $d^2$ isogeny to a product of elliptic curves. We call such surfaces $d$-elliptic, and denote the (normalization of the) locus of $d$-elliptic surfaces by $\hh_d$. A curve of genus $2$ is $d$-elliptic in the usual sense, i.e.\ admits a degree $d$ covering onto an elliptic curve, if and only if its Jacobian is $d$-elliptic in this sense. 

These loci of $d$-elliptic curves and surfaces are classically studied by algebraic geometers and number theorists. Biermann and Humbert showed that the locus of $d$-elliptic surfaces in $\A_2$ is exactly equal to the Humbert surface \cite{vdG88} of invariant $d^2$, see \cite{kani94}. Moreover, a natural double cover of $\hh_d$ can be described as a quotient $\Gamma \setminus \mathfrak{H}\times \mathfrak H$ with a group $\Gamma$ acting by a ``twisted'' diagonal action, which makes the double cover appear as a degenerate Hilbert modular surface. These degenerate Hilbert modular surfaces were studied in \cite{hermann91}, and in \cite{carlton01} steps were taken towards studying modular forms on them. The latter article gives a concrete interpretation to the spaces of modular forms that we find to occur in the cohomology of these local systems.

We now give an outline of the article. In section 2, we define the spaces $\hh_d$ via their modular interpretation, and explain their description as quotients of products of modular curves. We also discuss how representations of $\SL_2(\Z/d)$ behave under conjugation by elements of $\GL_2(\Z/d)$. Section 3 proves branching formulas for $\Sp_2 \wr \sym_2 \hookrightarrow \Sp_4$  and $\Sp_2 \times \sym_2 \hookrightarrow \Sp_2 \wr \sym_2$, which will later be used to determine how certain local systems behave under pullback between modular varieties. These local systems are introduced in Section 4. In section 5, we recall the Eichler--Shimura theory, expressing the cohomology of such local systems on modular curves in terms of modular forms, and show how this leads to a description of the cohomology of local systems on $\hh_d$. In section 6, we specialize to $d=2$ and show how the results of this paper can be used to compute the Euler characteristic of the space of $n$-pointed bi-elliptic genus $2$ curves in the Grothendieck group of Hodge structures for any $n$. 




I am grateful to my advisor Carel Faber for posing this problem to me, and for patient discussions. 

\begin{convention}Unless stated otherwise, all cohomology will be taken in the category of rational mixed Hodge structures. 
\end{convention}

\begin{rem} Restricting to Hodge structures is not really necessary. We could for instance substitute ``smooth $\ell$-adic sheaf'' for ``local system'' throughout and our computations would work equally well in the category of $\ell$-adic Galois representations and positive characteristic (provided that the integer $d$ is invertible on our base scheme).  In fact it is not so hard (although we shall not do so) to do everything motivically, using the results of \cite{scholl90} to construct these cohomology groups as Chow motives. \end{rem}

\section{The $d$-elliptic loci}

\begin{defn} Let $(A,\Theta)$ be a principally polarized abelian surface. We say that $A$ is $d$-elliptic if there is a (connected) curve $E$ lying on $A$ such that: 
\begin{enumerate}
\item $E$ is a subgroup of $A$ under the group law; 
\item the genus of $E$ is one; 
\item $(E.\Theta) = d$. 
\end{enumerate} \end{defn} 

There are several equivalent characterizations of $d$-elliptic surfaces, which we now briefly recall. This description is due to Frey and Kani \cite{fk91}, who write that  ``the following construction appears to be known in principle''. Much of this section is a special case of the general theory in \cite[Chapter 12]{bl92}. See also \cite{kani94}. 

Let $A$ be $d$-elliptic. Then $E \hookrightarrow A$ dualizes to a surjection $A \to E$ whose kernel is connected. Then it, too, is an elliptic curve, which we denote $E'$ and call the \emph{conjugate} of $E$. In other words, $E'$ is the ``Prym variety'' of $A \to E$. The curves $E$ and $E'$ intersect (inside $A$) exactly in their respective $d$-torsion points. The induced isomorphism $\phi \colon E[d]\to E'[d]$ of $d$-torsion groups inverts the Weil pairing, i.e. $\la x, y \ra_E = \la \phi x, \phi y \ra^{-1}_{E'}$. It follows that the induced isogeny $E \times E' \to A$ has degree $d^2$, and that it is defined by quotienting out the subgroup defined by the graph of $\phi$. Hence one could also define an abelian surface to be $d$-elliptic when it can be written as $E \times E' / \{ (x,\phi x) \st x \in E[d]\} $, where $\phi$ is an isomorphism that inverts the Weil pairing as above.

\begin{defn}\label{unordered}We denote by $\hej_d$ the moduli stack of pairs $(A,\Theta, \{E,E'\})$ where $(A,\Theta)$ is $d$-elliptic abelian surface and $\{E,E'\}$ is an unordered pair\footnote{To be more precise, one should not consider unordered pairs but rather the groupoid whose objects are 4-tuples $(A,\Theta,E,E')$ satisfying the above conditions, and whose isomorphisms are cartesian diagrams which are allowed to switch $E$ and $E'$.} of conjugate elliptic subgroups of $A$. \end{defn}

The preceding paragraph implies an alternative description of $\hej_d$. Let $Y(d)$ denote the open modular curve parametrizing elliptic curves with a symplectic basis of their $d$-torsion groups. Then clearly the space
\[ Y(d) \times Y(d) / (\sym_2 \times \SL_2(\Z/d)), \]
where $\sym_2$ swaps the two factors and $\SL_2(\Z/d)$ acts diagonally, parametrizes unordered pairs of elliptic curves together with a symplectic isomorphism of their $d$-torsion groups. To invert the Weil pairing we need to consider instead an action of the semidirect product
$ \sym_2 \ltimes \SL_2(\Z/d) $
%
%
where $\sym_2$ acts on $\SL_2(\Z/d)$ by conjugation with an element $\varepsilon \in \mathrm{GL}_2(\Z/d)$ such that $\det(\varepsilon) = -1$. Then one has the following proposition:

\begin{prop} There is an isomorphism
\[ Y(d) \times Y(d) / (\sym_2 \ltimes \SL_2(\Z/d)) \cong \hej_d\] 
where $\SL_2(\Z/d)$ acts normally on the first copy of $Y(d)$ and via the conjugated action on the second copy. \end{prop}

\begin{rem}The natural map from $\hej_d$ to the locus of $d$-elliptic abelian surfaces in $\A_2$ exhibits $\hej_d$ as the normalization of the $d$-elliptic locus. \cite[Corollary 3.10.]{kani94}\end{rem}

\begin{rem}When $A = \mathrm{Jac}(C)$, the composition $C \to A \to E$ is a covering of degree $d$ which is minimal in the sense that it does not factor through an isogeny; conversely, any such covering $C \to E$ induces a map $E \to \mathrm{Jac}(C)$ making the Jacobian $d$-elliptic. \end{rem}

%

%
\newcommand{\Rep}{\mathsf{Rep}}

We shall need to see how the action of $\sym_2$ on $\SL_2(\Z/d)$ defined above acts on representations of $\SL_2(\Z/d)$. 

\begin{notation}If $V$ is a representation of $\SL_2(\Z/d)$, then $V^\varepsilon$ denotes the representation obtained by conjugation by $\varepsilon \in \GL_2(\Z/d)$ of determinant $-1$. \end{notation}

It is clear that $V^{\varepsilon \varepsilon} \cong V$, and that $V^{\varepsilon}$ does not depend on the choice of $\varepsilon$ up to isomorphism.

\begin{lem}Let $p$ be a prime, $\Z_p$ the $p$-adic integers, and choose $A \in \SL_2(\Z_p)$. There is always a matrix $\varepsilon \in \GL_2(\Z_p)$ with $\det(\varepsilon) = -1$ such that $A = \varepsilon A^{-1} \varepsilon^{-1}$. \end{lem}

\begin{proof}Let $A = (\begin{smallmatrix} a & b \\ c & d \end{smallmatrix})$ and put $\varepsilon = (\begin{smallmatrix} x & y \\ z & -x \end{smallmatrix})$. 
One checks that 
\[ \begin{pmatrix} x & y \\ z & -x \end{pmatrix}\begin{pmatrix} a & b \\ c & d \end{pmatrix} = \begin{pmatrix} d & -b \\ -c & a \end{pmatrix}\begin{pmatrix} x & y \\ z & -x \end{pmatrix} \] 
holds if and only if
\[ (a-d)x + cy - bz = 0.\]
By multiplying this equation by an appropriate factor $p^\lambda$, we may assume that at least one of $(a-d)$, $b$ and $c$ is a $p$-adic unit. (If not, $A$ is plus or minus the identity and we are done.) If $b$ is a unit, we put $x = 1$, $y = 0$ and $z = b^{-1}(a-d)$. Then $\det(\varepsilon) = -1$ and we are done. Similarly if $c$ is a unit. If $(a-d)$ is a unit and $b$ and $c$ are not, then substitute 
\[ x = (a-d)^{-1}(bz-cy)\]
into the equation 
\[ -x^2 - yz = -1\]
for the determinant. If $p> 2$, then reducing the resulting equation modulo $p$ gives the equation $- yz = -1$.  For $p=2$ we must reduce modulo $8$, and find
\[4\delta_1 y^2 + 4\delta_2 z^2 - yz = -1\]
for some $\delta_1,\delta_2 \in \{0,1\}$. Either way one can now check that there is a solution in $\Z_p$ by a version of Hensel's lemma. 
\end{proof}

\begin{rem}The proposition is false if we instead put the condition $\det(\varepsilon) = 1$. For instance, the matrix $(\begin{smallmatrix} 1 & 1 \\ 0 & 1\end{smallmatrix})$ is not conjugate in $\SL_2(\Z_p)$ to its inverse if $p \equiv 3 \pmod 4$. \end{rem}

\begin{prop}\label{sl2}Let $V$ be any representation of $\SL_2(\Z/n)$. Then $V^\varepsilon$ is isomorphic to the dual (contragredient) of $V$. \end{prop}

\begin{proof}By the Chinese remainder theorem, we may assume that $n = p^\lambda$ is a prime power. Let $\chi$ be the character of $V$. Then the character of its dual is $g \mapsto \chi(g^{-1})$, and the character of $V^\varepsilon$ is $g \mapsto \chi(\varepsilon g \varepsilon^{-1})$. But $g^{-1}$ and $\varepsilon g \varepsilon^{-1}$ lie in the same conjugacy class of $\SL_2(\Z/p^\lambda)$ by the preceding lemma, so the two characters coincide. \end{proof}
\newcommand{\PGL}{\mathrm{PGL}}
\newcommand{\PSL}{\mathrm{PSL}}
\begin{ex}Let $p$ be an odd prime. In this case one can quite easily see the preceding proposition concretely from the character table of $\SL_2(\Z/p)$. Note that an element $\varepsilon \in \GL_2(\Z/p)$ can act nontrivially  by conjugation on the representations of $\SL_2(\Z/p)$ only if it is nonzero in $\PGL_2(\Z/p)/\PSL_2(\Z/p) \cong \sym_2$, i.e.\ when $\det(\varepsilon)$ is a nonsquare in $\Z/p$. The character table of $\SL_2(\Z/p)$ is constructed in \cite[Section 5]{fh91}. Their construction also shows that all but four exceptional irreducible representations are restrictions of representations from $\GL_2(\Z/p)$, hence invariant under conjugation and isomorphic to their duals (since every element of $\GL_2(\Z/p)$ is conjugate to its inverse). The remaining four occur when restrictions from $\GL_2(\Z/p)$ split into two irreducibles under restrictions, so they are pairwise switched by conjugation by $\varepsilon$ precisely when $\det(\varepsilon)$ is a nonsquare. On the other hand the entries of the character table for these four representations contain a square root of the Legendre symbol ${( \frac {-1} p) }$ as their only potentially non-real entries.\end{ex}


\section{Branching formulas}

\renewcommand{\V}{V}
\renewcommand{\R}{U}
\renewcommand{\W}{W}
Recall that irreducible finite-dimensional representations of $\Sp_{2g}$ are indexed by their highest weight, which is a decreasing sequence $l_1 \geq \cdots \geq l_g \geq 0$ of integers. 
The corresponding irreducible representation appears for the ``first'' time inside 
\[ \Sym^{l_1-l_2}(\wedge^1 V) \otimes \Sym^{l_2-l_3}(\wedge^2 V)\otimes \cdots \otimes \Sym^{l_{g-1}-l_g}(\wedge^{g-1} V)\otimes \Sym^{l_g}(\wedge^g V), \]
where $V$ is the defining $2g$-dimensional representation of $\Sp_{2g}$. For example, the weight vector $l \geq 0 \geq \cdots \geq 0$ corresponds to the irreducible representation $\Sym^l V$. In particular, all irreducible representations of $\Sp_2$ are symmetric powers of the defining one. 

\begin{notation}We denote the irreducible representation of $\Sp_4$ with highest weight $l\geq m\geq 0$ by $W_{l,m}$. For integers $l,m$ which do not satisfy $l \geq m \geq 0$, we put $\W_{l,m}=0$. We similarly index the irreducible representations of $\Sp_2$ as $V_a$. \end{notation}

The wreath product $\Sp_2 \wr \sym_2 = (\Sp_2\times \Sp_2)\rtimes \sym_2$ embeds naturally in $\Sp_4$ as the subgroup which preserves a decomposition of a 4-dimensional symplectic vector space into a sum of two unordered symplectic subspaces. We now determine a branching rule for this inclusion. First we need a description of the irreducible representations of $\Sp_2 \wr \sym_2$.

\begin{defn}Let $a, b \in \Z$. Define $\R_{a,b}$ to be the representation of $\Sp_2 \wr \sym_2$ given by 
\[ \V_a \otimes \V_b \oplus \V_b \otimes \V_a \]
as a representation of $\Sp_2 \times \Sp_2$, with an $\sym_2$-action given by 
\[ \sigma(x \otimes y,y' \otimes x') = (x' \otimes y',y \otimes x), \]
where $\sigma = (12) \in \sym_2$. Secondly, for any $a \ge 0$, we define two representations $\R_a^+$ and $\R_a^{-}$ by giving
the $\Sp_2 \times \Sp_2$-representation $ \V_a \otimes \V_a $ the $\sym_2$-actions
\[ \sigma(x \otimes y) = y \otimes x \quad \text{ and } \quad \sigma(x \otimes y) = -y \otimes x,\]
respectively.  \end{defn}

\begin{prop} The representations $\R_{a,b}$, where $a \neq b$, and $\R_a^\pm$ are the only irreducible representations of $\Sp_2 \wr \sym_2$.  \end{prop}

\begin{proof} The representation theory of a semidirect product tells us that every irreducible representation of $\Sp_2 \wr \sym_2$ occurs in a representation induced from an irreducible of $\Sp_2 \times \Sp_2$. The irreducible $\V_{a}\otimes \V_b$ induces to $\R_{a,b}$, whereas the irreducible $\V_a \otimes \V_a$ induces to the sum $\R_{a,a}=\R_a^+ \oplus \R_a^{-}$. \end{proof}


\begin{prop} \label{branch1} Assume $l+m$ is even. Then the restriction of the representation $\W_{l,m}$ of $\Sp_4$ decomposes as 
\[ \Res_{\Sp_2\wr \sym_2}^{\Sp_4} W_{l,m}= \bigoplus_{0\leq i \leq m} \left( \R_{(l-m)/2 + i}^{(-1)^m} \oplus  \bigoplus_{0 \leq j < (l-m)/2}\R_{l-m+i-j,i+j} \right), \]
where $(-1)^m$ denotes `+' if $m$ is even and `--' if $m$ is odd. If $l+m$ is odd, then 
\[ \Res_{\Sp_2\wr \sym_2}^{\Sp_4} W_{l,m}= \bigoplus_{0\leq i \leq m} \bigoplus_{0 \leq j < (l-m)/2}\R_{l-m+i-j,i+j}. \]
\end{prop}

\begin{proof}We prove this by induction on $m$. Consider first $m=0$. It is clear that the restriction of $W_{1,0}=W$ is the representation $\R_{1,0} = V \oplus V$. The restriction of $W_{l,0}$ is then $\Sym^l (V \oplus V) = \bigoplus_{i+j=l}V_i \otimes V_j$, which agrees with the formula above.

%
%
For the induction step, we use the formula
\[ \W_{l,m}\otimes W_{1,0} = W_{l,m+1} \oplus W_{l+1,m} \oplus W_{l,m-1} \oplus W_{l-1,m},\]
a special case of Pieri's rule for the symplectic group. (This rule states that instead of adding a horizontal $k$-strip in all possible ways, as one would do for $\GL_n$, one should first remove a vertical $i$-strip and then add a horizontal $(k-i)$-strip, for all $0 \leq i \leq k$.) One can then prove that the right hand sides in the statement of Proposition \ref{branch1} satisfy the same behavior upon tensoring with $\R_{1,0}$, as it is easy to see that 
\[ U_{a,b}\otimes U_{1,0} = U_{a+1,b} \oplus U_{a,b+1} \oplus U_{a-1,b} \oplus U_{a,b-1} \]
(where $U_{a,a} = U_a^{+}\oplus U_a^{-}$) and 
\[ U_{a}^{\pm} \otimes U_{1,0} = U_{a+1,a} \oplus U_{a,a-1}.\]
This finishes the proof. \end{proof}

\begin{rem}If we were not interested in a closed formula but only in being able to compute the branching algorithmically, we could also have argued as Bergstr\"om and van der Geer do in \cite[Section 7]{bergstromvandergeer} for $\Sp_2 \wr \sym_3 \hookrightarrow \Sp_6$. \end{rem}


When studying bielliptic curves, we will need a second branching formula, now for $\Sp_2 \times \sym_2$ sitting diagonally inside $\Sp_2 \wr \sym_2$.

Branching for $\R_{a,b}$ is easy: as it consists of two isomorphic $\sym_2$-invariant and anti-invariant parts, when restricted to the diagonal we find an invariant and an anti-invariant copy of $\V_a \otimes \V_b$. 

To describe the $\R_a^{\pm}$, let us introduce some notation. Let $\V_a^{+}$ and $\V_a^{-}$ denote the representation $\V_a$ tensored with the trivial and sign representations of $\sym_2$, respectively. The representations $\V_a$ of $\Sp_2$ are multiplied according to the rule
\[ \V_a \otimes \V_b = \V_{a+b} \oplus \V_{a+b-2} \oplus \cdots \oplus \V_{a-b} \]
if $a \ge b$ by  \cite[Chapter  11]{fh91}. When $a=b$, we find that the summands on the right hand side alternate between the trivial and sign representations:

\begin{prop} \label{diagbranch}The representations $\R_a^\pm$ decompose as
\[ \Res^{\Sp_2 \wr \sym_2}_{\Sp_2 \times \sym_2}\R_a^{+} = \bigoplus_{k=0}^a \V_{2a}^{(-1)^k}\]
and
\[  \Res^{\Sp_2 \wr \sym_2}_{\Sp_2 \times \sym_2} \R_a^{-} = \bigoplus_{k=0}^a \V_{2a}^{(-1)^{k+1}},\]
respectively, where again $(-1)^k$ denotes `+' if $k$ is even and `--' if $k$ is odd. \end{prop}

\begin{proof}It suffices to consider $\R_a^{+}$. We will begin by decomposing the representation $\V_a$ into weight spaces. 
The case of $\Sp_2$ is particularly simple as we can replace our Cartan subalgebra with a single element,
\[ H = \begin{pmatrix} 1 & 0 \\ 0 & -1 \end{pmatrix} \in \mathfrak{sp}_2. \]Decomposing $\V_a$ into 1-dimensional eigenspaces for $H$, one finds \cite[Chapter  11]{fh91}
\[ \V_a = E_{a} \oplus E_{a-2} \oplus \cdots \oplus E_{-a}. \]
Since $\R_a^{+} = \V_a \otimes \V_a$ as an $\Sp_2$-representation, we get a similar eigenspace decomposition of $\R_a^{+}$ into a sum of copies of $E_i \otimes E_j$. We see that $\sym_2$ acts trivially on all eigenspaces of the form $E_i \otimes E_i$, whereas the subspaces of the form $E_i \otimes E_j \oplus E_j \otimes E_i$, $i \neq j$, split into two isomorphic subspaces with trivial and sign representation of $\sym_2$ respectively. This determines the decomposition of $\R_a^{+}$ into one-dimensional eigenspaces of $H$, along with their $\sym_2$-actions. Now one checks that the sum \[ \bigoplus_{k=0}^a \V_{2(a-k)}^{(-1)^k}\] has the same decomposition.  This determines the representations uniquely.
\end{proof}

\section{The relevant local systems}

\renewcommand{\W}{\mathbf{W}}
\renewcommand{\V}{\mathbf{V}}
\renewcommand{\R}{\mathbf{U}}

\begin{defn}Let $\W$ denote the \emph{standard local system} on $\A_2$, defined by
\[ \W = \mathrm{R}^1\pi_\ast \Q, \]
where $\pi : \mathcal{X} \to \A_2$ is the universal abelian surface. \end{defn}

Since $\pi$ is a smooth projective morphism, there is a natural variation of Hodge structure on $\W$. The fiber of $\W$ over a point $[(A,\Theta)]$ is canonically isomorphic to the $4$-dimensional symplectic vector space $H^1(A)$. The local system $\W$ can also be defined via the inclusion 
\[ \pi_1^{\mathrm{orb}}(\A_2) = \Sp_4(\Z) \subset \Sp_4 \]
and the natural action of $\Sp_4$ on $H^1(A)$. 

By pulling back $\W$ along the map $\hh_d \to \A_2$, we get a local system on $\hh_d$ which will also be denoted $\W$. There is a second natural way of writing down such a local system. There is an obvious forgetful map 
\[ \hh_d \to \A_1 \times \A_1 / \sym_2 \]
obtained by forgetting the isomorphism $\phi$ of $d$-torsion groups. Since $\A_1 \times \A_1 / \sym_2$ sits inside $\A_2$ as the complement of $\M_2$, we can pull back the local system $\W$ on $\A_2$ to $\hh_d$ also along this composition. Let us call the result $\widetilde \W$. 

\begin{prop}There is a natural isomorphism $\widetilde \W \to \W$. \end{prop}

\begin{proof}There are two universal families $\Y$ and $\X$ over $\hh_d$; the first is the universal product of two elliptic curves, and the second is the universal $d$-elliptic abelian surface. The graph of the isomorphism $\phi$ defines a finite flat group scheme $\mathcal Z$ in $\Y$ such that $\Y/\mathcal Z \cong \X$. The resulting map $\Y \to \X$ is fiberwise an isogeny, hence fiberwise an isomorphism on rational cohomology. Since $\widetilde \W = \mathrm{R}^1\pi_\ast \Q_\Y$ and $\W = \mathrm{R}^1\pi_\ast \Q_\X$, we conclude by base change. \end{proof}

We will henceforth not distinguish between $\W$ and $\widetilde \W$. 

Every irreducible representation $W_{l,m}$ of $\Sp_4$ induces naturally a local system $\W_{l,m}$ on $\A_2$. One way to see this is that we can apply the symplectic Schur functor corresponding to $W_{l,m}$ to the local system $\W$. The local system $\W_{l,m}$ is the same as the one obtained from the inclusion $\pi_1^{\mathrm{orb}}(\A_2) \subset \Sp_4$ and the representation $W_{l,m}$, but the construction with Schur functors shows that it carries a natural variation of Hodge structure. See \cite[Chapter VI, Section 5]{faltingschai} for another approach to constructing the mixed Hodge structure on the cohomology of $\W_{l,m}$.

There is similarly an inclusion 
\[ \pi_1^{\mathrm{orb}}((\A_1 \times \A_1 )/ \sym_2) = \Sp_2(\Z) \wr \sym_2 \subset \Sp_2 \wr \sym_2,\]
so for each of the representations $U_{a,b}$ and $U_a^{\pm}$ one gets a local system $\R_{a,b}$ respectively $\R_{a}^{\pm}$ on $(\A_1 \times \A_1) / \sym_2$. The pullback of $\W$ to $(\A_1 \times \A_1 )/ \sym_2$ is exactly the local system $\R_{1,0}$, and the pullbacks of the local systems $\W_{l,m}$ are determined by the branching formulas of Proposition \ref{branch1}. If we want to also consider these as variations of Hodge structure, then we need to add a Tate twist. For instance, the pullback of $\W_{2,1}$ is
\[ \R_{2,1} \oplus \R_{1,0}(-1);\]
these twists are easily put in ``by hand'' so that the pulled back expression is homogeneous. More conceptually, one could have worked with $\mathrm{GSp}_4$ instead of $\Sp_4$ from the beginning.

The preceding paragraph also describes the pullback of $\W$ to $\hh_d$, of course.



\section{Cohomology of local systems}\label{estheory}

\renewcommand{\S}{\mathbf{S}}
\newcommand{\E}{\mathbf{E}}

\begin{notation}For a congruence subgroup $\Gamma'$ of $\SL_2(\Z) = \Gamma$, we denote by $E_k(\Gamma')$  and $S_k(\Gamma')$ the spaces of Eisenstein series and cusp forms of weight $k$ for $\Gamma'$, respectively. \end{notation}

The cohomology of $\V_a$ on $Y(d)$ is described by the Eichler-Shimura isomorphism. For $a > 0$, $H^1(Y(d),\V_a)$ is the only nonzero cohomology group. It has a mixed Hodge structure whose nonzero bigraded pieces have Hodge numbers
\[ (a+1,0), \quad (0,a+1) \quad \text{and} \quad (a+1,a+1).\] Under the Eichler-Shimura isomorphism, these subquotients are interpreted as spaces of modular forms for $\Gamma(d)$: they are the holomorphic cusp forms of weight $a+2$, their antiholomorphic complex conjugates, and Eisenstein series of weight $a+2$, respectively. Let us {define} 
\[ \S_{a+2}(\Gamma(d)) = \mathrm{gr}^W_{a+1} H^1(Y(d),\V_{a}) \]
and 
\[ \E_{a+2}(\Gamma(d)) = \mathrm{gr}^W_{2a+2} H^1(Y(d),\V_{a}).\]
Tensoring with $\C$, one has that 
\[ \S_a(\Gamma(d))_\C = S_a(\Gamma(d)) \oplus \overline{S_a(\Gamma(d))}\]
and 
\[ \E_a(\Gamma(d))_\C = E_a(\Gamma(d)).\]
The extension is in fact trivial, i.e.
\[ H^1(Y(d),\V_{a}) = \S_{a+2}(\Gamma(d))\oplus \E_{a+2}(\Gamma(d)),\]
see \cite{gorinov1}. The cusp forms are the arithmetically interesting classes, as the Eisenstein series (when $d \geq 3$ and $a > 0$) are simply given by 
\[ \E_{a+2}(d) \cong H^0(X(d) \setminus Y(d))(-a-1).\]
This uses that $\Gamma(d)$ has no irregular cusps. When $a =0$ one needs to subtract a copy of the trivial representation from the right hand side, and when $d \leq 2$, the above statement only holds when $a$ is even. To get rid of the Eisenstein series, one can also consider the \emph{inner} or \emph{parabolic cohomology}, which is defined to be the image of the compactly supported cohomology inside the ordinary cohomology, and is denoted $H^\ast_!$. There one has that
\[ H^\ast_!(Y(d),\V_a) = H^1_!(Y(d),\V_a) = \S_{a+2}(\Gamma(d))\]
for any $a\geq0$.

\begin{notation}If $\rho$ is an irreducible representation of a group $G$, and $V$ is any representation of $G$, then we denote by $V^{(\rho)} = \mathrm{Hom}_G(\rho,V)$. In other words, $V = \oplus_\rho \rho \otimes V^{(\rho)}$ is the decomposition of $V$ into irreducibles.  \end{notation}


\begin{thm} \label{coh}Let $K$ be a splitting field of $\SL_2(\Z/d)$ over $\Q$. Keep notation as above. Then
\[ H^2_!(\hh_d,\R_{a,b})_K = \bigoplus_{\rho} \S_{a+2}(\Gamma(d))_K^{(\rho)}\otimes \S_{b+2}(\Gamma(d))_K^{(\rho)}, \]
\[ H^2_! (\hh_d, \R_a^+)_K =  \bigoplus_{\rho} \wedge^2 \S_{a+2}(\Gamma(d))_K^{(\rho)} \]
and
\[ H^2_! (\hh_d, \R_a^{-})_K =  \bigoplus_{\rho} \Sym^2 \S_{a+2}(\Gamma(d))_K^{(\rho)},\]
where all sums are taken over the set of isomorphism classes of irreducible representations of $\SL_2(\Z/d)$.
\end{thm}

\begin{proof}We tensor with $K$ only so that the resulting mixed $K$-Hodge structures admit decompositions into absolutely irreducible representations of $\SL_2(\Z/d)$, and we shall from now on omit $K$ from the notation. Motivically, this corresponds to considering motives with coefficients in $K$. 

Since we work with rational coefficients, we may 
compute the cohomology of these local systems on $Y(d) \times Y(d)$ and take $(\SL_2(\Z/d) \rtimes \sym_2)$-invariants. Consider first $\R_{a,b}$. One has 
\[ H^\ast(Y(d)^{\times 2},\R_{a,b}) = H^\ast(Y(d),\V_a) \otimes H^\ast(Y(d),\V_b) \oplus H^\ast(Y(d),\V_b) \otimes H^\ast(Y(d),\V_a)\]
by the K\"unneth formula. Let us first take $\SL_2(\Z/d)$-invariants. Schur's lemma implies that whenever $V$ and $W$ are irreducible representations of a group $G$, then the trivial representation occurs with multiplicity $1$ in $V \otimes W$ if $V$ and $W$ are duals, and does not occur otherwise. It follows then from Proposition \ref{sl2} that the $\SL_2(\Z/d)$-invariants of $H^\ast(Y(d)^{\times 2},\R_{a,b})$ are given by
\[ \bigoplus_{\rho} H^\ast(Y(d),\V_a)^{(\rho)} \otimes H^\ast(Y(d),\V_b)^{(\rho)} \oplus H^\ast(Y(d),\V_b)^{(\rho)} \otimes H^\ast(Y(d),\V_a)^{(\rho)}, \]
since the action of $\SL_2(\Z/d)$ was twisted by $\varepsilon$ on the second factor. This, in turn, splits into two isomorphic subspaces, one $\sym_2$-invariant and one anti-invariant. Clearly the inner cohomology of $\V_a \otimes \V_b$ on $Y(d) \times Y(d)$ is the tensor product of the respective inner cohomologies, and the result follows.

For $\R_a^\pm$, one starts out arguing as above, and finds that one must determine the $\sym_2$-invariant, resp. anti-invariant, subspace of
\[ \bigoplus_{\rho}H^1_!(Y(d),\V_a)^{(\rho)} \otimes  H^1_!(Y(d),\V_a)^{(\rho)}. \]
Since these are odd cohomology classes, it follows from the presence of the Koszul sign rule in the K\"unneth isomorphism that the alternating tensors are $\sym_2$-invariant whereas the symmetric tensors are anti-invariant. This finishes the proof. \end{proof}

\begin{rem}It is clear from the proof that we can also compute the non-parabolic cohomology of $\R_{a,b}$ and $\R_{a}^\pm$ in much the same way. One gets only a slightly more complicated statement of the theorem, involving both Eisenstein series and the ``extra'' nonzero cohomology group that one gets for $a=0$: one has $H^0(Y(d),\V_0) = H^0(Y(d)) = \Q$. \end{rem}

\begin{rem}Using Proposition \ref{sl2} and applying Schur's lemma as in the preceding proof, one shows the following statement. Let $V$ be any representation of $\SL_2(\Z/d)$, and let $V^A$ be the representation obtained by conjugation with an element $A \in \GL_2(\Z/d)$. Then the dimension of $(V \otimes V^A)^{\SL_2(\Z/d)}$ is maximized when $\det A = -1$. (One needs to apply the inequality $x^2 + y^2 \geq 2xy$.) By taking $V = H^1(Y(d),\mathcal O_{Y(d)})$ one recovers a theorem of Carlton \cite[Corollary 5.4]{carlton01}, that the geometric genus of a quotient $Y(d)\times Y(d) / \SL_2(\Z/d)$, where $\SL_2(\Z/d)$ acts normally on the first factor and by a conjugated action on the second factor, is maximized when one conjugates with a matrix of determinant $-1$. Carlton's proof is rather different and uses an analogue of Atkin--Lehner theory on such quotients. \end{rem}

\begin{rem}When $a$ is even and $d \geq 3$, the dimension of each isotypical component of $S_{a}(\Gamma(d))$ can be found in \cite[Theorem 3.4.3]{weinstein07}. \end{rem}


\section{The case $d=2$; pointed bi-elliptic curves}

We now focus on the case $d=2$ as a special case of the theory from above. Strictly speaking this case is a bit easier. Recall that the condition that the Weil pairing was to be inverted meant that we had to consider a conjugation action of $\sym_2$ on $\SL_2(\Z/d)$ by a matrix $\varepsilon$ of determinant $-1$. Over $\Z/2$, we can take $\varepsilon = \id$, so the semidirect product is in fact a direct product. Moreover, the isomorphism $\SL_2(\Z/2) \cong \sym_3$ makes the representation theory very simple. 

Another minor difference arises because $Y(2)$ is a stack, not a scheme: the elliptic involution fixes the $2$-torsion of any elliptic curve. The elliptic involution acts as multiplication by $(-1)^a$ on the fibers of the local system $\V_a$ on $Y(2)$, so the cohomology of this local system vanishes unless $a$ is even. Thus the local systems $\R_{a,b}$ have vanishing cohomology on $Y(2)\times Y(2)$, hence also on $\hh_2$, unless $a$ and $b$ are both even. Similarly $\R_a^\pm$ has vanishing cohomology unless $a$ is even. 

Let us compute the cohomology groups $H^\ast(Y(2),\V_a)$ as $\sym_3$-representations. The subgroups $\Gamma(2)$ and $\Gamma_0(4)$ are conjugate, so $Y(2) \cong Y_0(4)$. We shall  work with $Y_0(4)$, essentially because of Atkin--Lehner-theory. 

Let $s_3$, $s_{21}$ and $s_{111}$ denote the representations corresponding to the respective partitions, that is, the trivial, standard and sign representations, respectively. Since 
\[ Y_0(4)/\sym_3 \cong \A_{1}  \text{ and } Y_0(4)/\sym_2 \cong Y_0(2), \] (where $\sym_2$ is the subgroup generated by a transposition), it is not hard to see that decomposing the spaces of modular forms and cusp forms as $\sym_3$-representations is equivalent to determining which forms are newforms of the bigger groups $\Gamma_0(2)$ and the full modular group $\Gamma$. The $s_3$-part is exactly those which are modular forms for $\Gamma$, the $s_{21}$-part corresponds to those which are lifted from newforms for $\Gamma_0(2)$ (we get a two-dimensional subspace of oldforms for $\Gamma_0(4)$ from a one-dimensional space of newforms for $\Gamma_0(2)$ since there are two different liftings), and the $s_{111}$-part consists of the newforms. This is implicitly used in \cite{bfg08}. We record this as a proposition:

\begin{prop} \label{blaha}We have $S_{a}(\Gamma(2))^{(s_3)} \cong S_a(\Gamma)$, $S_{a}(\Gamma(2))^{(s_{21})} \cong S_{a}^{\mathrm{new}}(\Gamma_0(2))$ and $S_{a}(\Gamma(2))^{(s_{111})} \cong S_{a}^{\mathrm{new}}(\Gamma_0(4))$. \end{prop}

In particular, we can determine the decomposition of $S_a(\Gamma(2))$ from the respective dimension formulas for $\Gamma_0(4), \Gamma_0(2)$ and $\Gamma$, see e.g.\ \cite{diamondshurman}. We omit the details. 



\begin{rem} This result corrects a minor error in \cite[Theorem 3.4.3]{weinstein07}. Weinstein gives a formula for how $S_a(\Gamma(d))$, $d \ge 2$ and $a$ even, decomposes into irreducible representations. However, Weinstein's formula is only correct for $d > 2$: his calculation is an equivariant version of the usual derivation of the dimension formula for $\Gamma(d)$, which needs to be modified when $d=2$ since $-1 \in \Gamma(2)$. The correct statement of Weinstein's result when $d=2$ is easily deduced from Proposition \ref{blaha}. \end{rem}

If we are interested in the locus of $d$-elliptic curves in $\M_2$ instead of $d$-elliptic abelian surfaces in $\A_2$, we need to remove those pairs of elliptic curves which map into $(\A_{1}\times \A_1) / \sym_2$ inside $\A_2$. A description of this locus has been worked out by Frey and Kani \cite{kani91}, who show that it is a union of graphs of Hecke correspondences on $Y(d) \times Y(d)$. The special case of $d=2$ becomes very simple: here, we simply find the diagonal inside $Y(2) \times Y(2)$. So for bi-elliptic curves, we need to understand the cohomology of the local systems on the diagonal. 

%



Let $\Delta$ denote the image in $\hh_2$ of the diagonal substack of $Y(2) \times Y(2)$. Note that the diagonal is invariant under the action of $\SL_2(\Z/2) \times \sym_2$. 

\begin{prop} \label{diag}Let $a$ and $b$ be even integers. Then
\[ H^\ast(\Delta,\R_{a,b}) = H^\ast(\A_{1}, \V_a \otimes \V_b), \]
and the cohomology vanishes if either $a$ or $b$ is odd. Similarly, \[ H^\ast(\Delta,\R_a^{+}) = \bigoplus_{k=0}^{a/2}H^\ast(\A_{1},\V_{4k})\]
and
\[ H^\ast(\Delta,\R_a^{-}) = \bigoplus_{k=1}^{a/2}H^\ast(\A_{1},\V_{4k-2})\]
when $a$ is even, and the cohomology vanishes otherwise.
 \end{prop}

\begin{proof}We have already seen the vanishing part of the proposition. We do calculations on the diagonal inside $Y(2) \times Y(2)$ and take invariants. The pullback of $\R_{a,b}$ and $\R_a^\pm$ to the diagonal in $Y(2) \times Y(2)$ is determined by the branching formulas in Proposition \ref{diagbranch} and the remarks preceding it. It is also clear from these formulas what happens when we take $\sym_2$-invariants. Taking $\SL_2(\Z/2)$-invariants just corresponds to descending to level 1, so we get cohomology of local systems on $\A_{1}$. \end{proof} 

\begin{rem}Note that the proposition above is false if we do not demand that $a$ and $b$ are even: if $a$ and $b$ are odd, then $H^\ast(\Delta,\R_{a,b}) = 0$, but $H^\ast(\A_{1}, \V_a \otimes \V_b)$ is nonzero. This reflects the fact that the diagonal embedding $Y(2) \to Y(2) \times Y(2)$ is not an isomorphism onto the diagonal substack, unlike when $d \geq 3$. \end{rem}


Let $\mathcal{B}_n$ be the moduli space parametrizing bi-elliptic curves of genus two with $n$ distinct marked points. Here we define a bi-elliptic curve of genus two to be a curve $C$ together with an unordered pair of conjugate double covers $C \to E$, $C\to E'$, cf.\ Definition \ref{unordered}.

As we shall now see, the knowledge of the cohomology of the local systems $\W_{l,m}$ can be used to compute the $\sym_n$-equivariant Euler characteristic (in the Grothendieck group of mixed Hodge structures) of $\mathcal{B}_n$. It will be more convenient to switch to compactly supported cohomology at this point. We need some general results on relative configuration spaces due to Getzler \cite{getzler99}. In Getzler's setting, one considers a quasi-projective morphism of varieties $\pi \colon \X \to \M$ and the \emph{relative configuration space} $F(\X/\M,n)$, which is the complement of the ``big diagonal'' in the $n$th fibered power of $\X$ over $\M$. Getzler proves the formula 
\begin{equation*} \label{getzler} \sum_{n \geq 0} \e_\M^{\sym_n}(F(\X/\M,n),\Q) = \prod_{k \geq 0} (1+p_k)^{\frac{1}{k}\sum_{d|k}\mu(k/d) \psi_d( \e_\M(\X,\Q))},\end{equation*}
where $\e_\M(\X,\Q)$ denotes the compactly supported relative Euler characteristic obtained from $\mathrm{R}\pi_!\Q$  in the Grothen\-dieck group of the bounded derived category of mixed Hodge modules on $\M$; where $\e^{\sym_n}_\M(F(\X/\M,n),\Q)$ is similarly the $\sym_n$-equivariant compactly supported Euler characteristic, taken in the same Grothen\-dieck group but tensored with the ring $\Lambda$ of symmetric functions, i.e.\ the sum of the representation rings of $\sym_n$ for all $n$;  where the $p_k$ are power sums; where $\psi_d$ denotes an Adams operation; where $\mu$ denotes the M\"obius function; and the factors on the right hand side should be formally expanded as binomial series. To obtain the actual Euler characteristic from this formula, one needs to take the proper pushforward of both sides along $\M \to \rm{Spec } \; \C$, which produces an equality in the usual Grothendieck group of mixed Hodge structures tensored with $\Lambda$. 

In our case, we put $\M = \hh_2 \times_{\A_2} \M_2$, the moduli space of bi-elliptic curves, and $\pi \colon \X\to \M$ its universal family of genus two curves. So $\M = \mathcal{B}_0$, $\X = \mathcal{B}_1$, and $F(\X/\M,n) = \mathcal{B}_n$. One has \[\e_\M(\X,\Q) = \Q - \W + \Q(-1).\] Expanding the right hand side of Getzler's formula yields an expression where the coefficient before each monomial in $\Lambda$ is a formal sum of certain Schur functors applied to $\W$. Decomposing these Schur functors into irreducible representations of the symplectic group allows us to re-write this as a sum of the local systems $\W_{l,m}$ with some Tate twists. It follows that the results so far in the article allow us to compute the $\sym_n$-equivariant Euler characteristic: one sees from the Gysin sequence that 
\[ \e(\M,\W_{l,m}) = \e(\hh_2,\W_{l,m}) + \e(\Delta,\W_{l,m}),\]
and the right hand side can be expressed in terms of the $\sym_3$-equivariant Euler characteristics $\e^{\sym_3}(Y(2),\V_a)$  by combining Propositions \ref{branch1}, \ref{coh}, \ref{diagbranch} and \ref{diag}. Finally, $\e^{\sym_3}(Y(2),\V_a)$ can be computed from the Eichler--Shimura theory quoted in Section \ref{estheory} and Proposition \ref{blaha}.

From the discussion above, one can calculate the $\sym_n$-equivariant Euler characteristic $\e^{\sym_n}(\mathcal{B}_n,\Q)$ for any $n$. The first few results are stated in Table \ref{tabell1}. We denote $\tate = H^2(\mathbf P^1)$; polynomials in $\tate$ with integer coefficients are interpreted in the natural way. The first occurrence of non-Tate cohomology is the $s_{111111}$-coefficient of $\e^{\sym_6}(\mathcal B_6,\Q)$, which is given by $\S_8(\Gamma_0(2)) -\tate^4 +3\tate +5.$

\begin{table}[t]
\begin{center}
\begin{tabular}{ | c || p{10cm} | }
\hline
  $n$ & $\e^{\sym_n}(\mathcal B_n,\Q)$  \\
\hline
0 & $\tate^2-\tate$  \\
1 & $(\tate^3-\tate)s_1$ \\
2 & $(\tate^4-\tate^2 +\tate)s_2 + (\tate^3-\tate^2-\tate+2)s_{11}$ \\
3 & $(\tate^5-2\tate^3+2\tate^2+\tate-2)s_3 + (\tate^4-\tate^3+2\tate)s_{21} + (-\tate^2+\tate+2)s_{111}$  \\
4 & $(\tate^6-2\tate^4+\tate^3+\tate^2-3\tate)s_4 + (\tate^5-2\tate^4+\tate^3+3\tate^2-\tate-2)s_{31} + $ \\
    & $(\tate^4-\tate^2-\tate+3)s_{22} + (-\tate^3+5\tate+2)s_{211} + (-\tate^3-\tate^2+\tate+3)s_{1111}$ \\
\hline 
\end{tabular}
\end{center}
\caption{Compactly supported Euler characteristic of $\mathcal B_n$ in the Grothendieck group of mixed Hodge structures.}
\label{tabell1}
\end{table}

The problem of computing these Euler characteristics was studied by means of point counts over finite fields in the author's M.Sc. project, using techniques similar to those of \cite{bergstrom09}. It was proven in this M.Sc. thesis that when $n \leq 5$ and $q$ is odd, the number of $\mathbf F_q$-points of $\mathcal{B}_n$ is given by a polynomial in $q$. Moreover, the calculations were done $\sym_n$-equivariantly. From this one obtains conjectural\footnote{The results of \cite{vdBE05} can not be applied in this case to yield an unconditional proof of these formulas for the Euler characteristics, both because of the restriction to odd $q$, and more seriously as the natural compactification of $\hh_2$ involves the modular curve $X(2)$, which is not a Deligne--Mumford stack unless the prime $2$ is invertible \cite{conrad05}.} formulas for the $\sym_n$-equivariant Euler characteristic of $\mathcal B_n$ in the Grothendieck group of $\ell$-adic Galois representations when $n \leq 5$. Needless to say, the results obtained there agree with the ones found here.

\bibliographystyle{petersen}

\bibliography{../database}

\end{document}